\documentclass[12pt,twoside]{amsart}
\usepackage{amsmath}
\usepackage{amssymb}
\usepackage{amscd}
\usepackage{xypic}
\usepackage[active]{srcltx}
\xyoption{all}
\title[Nilpotent $p$-local finite groups]
{Nilpotent $p$-local finite groups}

\author{Jos\'e Cantarero}
\thanks{The first author is partially supported by FEDER/MEC grant MTM2010-20692.}
\address{
\hfill\break Department of Mathematics\\
\hfill\break Stanford University\\
\hfill\break Stanford, California 94305\\
\hfill\break USA.}
\email{cantarer@stanford.edu}

\author{J\'er\^ome Scherer}
\thanks{The second author is partially supported by FEDER/MEC grant MTM2010-20692.}
\address{
\hfill\break SB MATHGEOM \\
\hfill\break EPFL\\
\hfill\break MA B3 455, Station 8\\
\hfill\break CH--1015 Lausanne\\
\hfill\break Switzerland.}
\email{jerome.scherer@epfl.ch}

\author{Antonio Viruel}
\thanks{The third author is partially supported by FEDER/MCI grant MTM2010--18089,
and Junta de Andaluc{\'\i}a grants FQM--0213 and P07--FQM--2863}
\address{
\hfill\break Departamento de \' Algebra, Geometr\'\i a y Topolog\'\i a\\
\hfill\break Universidad de M\' alaga\\
\hfill\break AP. 59\\
\hfill\break E-29080 M\' alaga\\
\hfill\break Spain.}
\email{viruel@agt.cie.uma.es}

\newcommand{\im}{\operatorname{\rm Im}\nolimits}
\newcommand{\Ker}{\operatorname{\rm Ker}\nolimits}
\newcommand{\A}{\ifmmode{\cal A}\else $\cal A$\fi}
\newcommand{\Ab}{\ifmmode{{\cal A}b}\else ${\cal A}b$\fi}

\newcommand{\F}{{\Bbb F}}

\newcommand{\Z}{{\Bbb Z}}

\newcommand{\Or}{{\mathcal O}}

\newcommand{\Ca}{{\mathcal C}}

\newcommand{\id}{\operatorname{id}\nolimits}

\newcommand{\tor}{\operatorname{Tor}\nolimits}
\newcommand{\Mor}{\operatorname{Mor}\nolimits}
\newcommand{\Inj}{\operatorname{Inj}\nolimits}
\newcommand{\Iso}{\operatorname{Iso}\nolimits}
\newcommand{\Hom}{\operatorname{Hom}\nolimits}
\newcommand{\Rep}{\operatorname{Rep}\nolimits}
\newcommand{\Inn}{\operatorname{Inn}\nolimits}
\newcommand{\Syl}{\operatorname{Syl}\nolimits}

\newcommand{\aut}{\operatorname{Aut}\nolimits}
\newcommand{\Out}{\operatorname{Out}\nolimits}

\renewcommand{\im}{\operatorname{Im}\nolimits}

\def\umono{\ar@{_{(}->}[u]}
\def\uumono{\ar@{_{(}->}[uu]}

\def\lmono{\ar@{_{(}->}[l]}
\def\llmono{\ar@{_{(}->}[ll]}

\newcommand{\p}{{\mathfrak{p}}}
\newcommand{\Ff}{{\mathcal{F}}}
\newcommand{\Ll}{{\mathcal{L}}}
\newcommand{\pcom}{^\wedge_p}

\newtheorem{theorem}{Theorem}[section]
\newtheorem{proposition}[theorem]{Proposition}
\newtheorem{corollary}[theorem]{Corollary}
\newtheorem{lemma}[theorem]{Lemma}
\newtheorem{definition}[theorem]{Definition}
\newtheorem{remark}[theorem]{Remark}



\begin{document}

\begin{abstract}
In this paper we provide characterizations of $p$-nilpotency
for fusion systems and $p$-local finite groups that are
inspired by known result for finite groups. In
particular, we generalize criteria by Atiyah, Brunetti, Frobenius, Quillen,
Stammbach and Tate.
\end{abstract}

\maketitle


\section*{Introduction}

This article is concerned with the concept of $p$-nilpotency. A
finite group $G$ is $p$-nilpotent if its $p$-Sylow subgroup $S$ has a
normal complement $N$ in $G$, that is, the composition $ S \to G \to
G/N $ is an isomorphism. The importance of this property
was highlighted by Henn and Priddy in \cite{MR1289332}:
for almost all finite $p$-groups $S$ and for $G$ a finite group with
$S$ as a Sylow $p$-subgroup, $G$ is $p$-nilpotent.

In the literature, there are characterizations of $p$-nilpotency using
different languages, such as group cohomology, Quillen categories, group theory, and fusion
systems. For example, the restriction map in mod $p$ cohomology $ H^*(G, \F_p)
\to H^*(S, \F_p) $ is an isomorphism if and only if $G$ is $p$-nilpotent
\cite{MR0318339}. The Frobenius $p$-nilpotency criterion in group theory, \cite[10.3.2]{MR648604},
characterizes these groups as those for which $\aut_G(P)$ is a $p$-group for
all subgroups $P$ of $S$.

For our purposes, an interesting characterization of $p$-nilpotency is given by \cite[Satz IV.4.9]{MR0224703}:
two subgroups of $S$ are conjugate in $G$ if and only if they are conjugate in $S$. In the
terminology of fusion systems \cite{MR1992826}, this is equivalent to saying that the fusion
system $\Ff_S(G)$, the category whose objects are subgroups of $S$ and morphisms are maps induced by
conjugations in $G$, equals $\Ff_S(S)$, the category with the same objects but
whose morphisms are maps induced by conjugations in $S$. This characterization
seems appropriate as a definition of $p$-nilpotency for fusion systems, and in particular,
for $p$-local finite groups. This has already been adopted by several authors, starting
with Kessar and Linckelmann in \cite{MR2379788}.

The objects known as $p$-local finite groups were introduced by
Broto, Levi, and Oliver in \cite{MR1992826} as a mean to extract the
essence of the $p$-local structure of a group of course, but also of
a block, or of more exotic occurrences such as the Solomon groups
\cite{MR1943386}. For a finite group $G$, this $p$-local structure
encodes the homotopical properties of the lattice of its
$p$-subgroups under the conjugation action of $G$. It is natural then that
most $p$-local properties enjoyed by finite groups also hold for
$p$-local finite groups.

Following this motivation and the definition for nilpotence of fusion systems
in \cite{MR2379788}, we say that a $p$-local finite group is nilpotent when its
fusion system is that of a $p$-group. Note that the prime $p$ is implicitly given
when we have a fusion system or a $p$-local finite group, so we omit the mention of
$p$. In fact, considering that a $p$-local finite group is an object created to
keep track of the $p$-local structure only, this naming convention agrees with the
well known fact that a finite group is nilpotent if and only it is $p$-nilpotent
for all prime numbers $p$.

According to the previous comments, if $G$ is a finite group with
Sylow subgroup $S$, the associated $p$-local group is nilpotent if and
only if $G$ is $p$-nilpotent. Note that this definition does not depend on the centric linking system $\Ll$,
but Proposition \ref{UniqueLinking} in this paper shows that the saturated fusion system $\Ff _S(S)$
has a unique centric linking system $\Ll _S(S)$ up to isomorphism. In particular, a $p$-local
finite group is nilpotent if and only if it is actually an honest  $p$-group. Sometimes it is more convenient to check
a topological condition than to stay at the level of the algebra of the
fusion data, and having a centric linking system allows for topological
characterizations.

We offer a list of characterizations of nilpotency in this context,
inspired by the work on $p$-nilpotency of honest groups. Apart from
the criteria we have already mentioned above, we also obtain analogues
of the criteria of Tate \cite{MR0160822}, Stammbach \cite{MR0473039},
Atiyah \cite{MR0318339}, and Quillen \cite{MR0318339}.

In the last decade this question has attracted a lot of attention. Let us
mention for example the work of Kessar and Linckelmann, who proved the
$p$-nilpotency theorem of Glaubermann and Thompson for blocks in
\cite{MR1929020} and later for fusion systems, \cite{MR2379788}.

Even more recently a result of Tate has been extended to fusion
systems by the four authors in \cite{DGPS}, providing another proof
of one part of our Theorem~\ref{quillen2}. In \cite{MR2448569},
the Frobenius nilpotency criterion, already present in \cite{MR2336638}
is used to prove Thompson's result in the setting of the theory of $p$-local
finite groups. This criterion is available in this article as Theorem~\ref{global}.

The work carried out in this paper is also of importance because nilpotency
is used to determine sparseness and extreme sparseness
of fusion systems, as defined in \cite{GL}. A fusion
system is sparse if the only proper subfusion system over the same Sylow
$p$-group $S$ is nilpotent. Sometimes sparseness implies
that the fusion system is constrained, a condition that reduces certain
questions about fusion systems to group theoretical ones. Extremely
sparse fusion systems are those for which the only proper subsystem
on any subgroup $Q$ of $S$ is nilpotent, and they are always
constrained.

Nilpotency is also used to define quasicentric subgroups, a concept heavily used in
\cite{MR2167090} and \cite{MR2302515}. A subgroup $P$ of $S$ is $\Ff$-quasicentric if
for all $P' \leq S $ that are $\Ff$-conjugate to $P$ and fully centralized in $\Ff$,
the fusion system $C_{\Ff}(P')$ is nilpotent. Some of the characterizations that we prove
in this paper can be useful to show the sparseness or extreme sparseness of a fusion
system, or to determine whether certain subgroups are quasicentric.

\medskip

The basic definitions of saturated fusion systems and their associated
centric linking systems are given in Section~1. Section 2 is dedicated to
generalities about nilpotent fusion systems and $p$-local finite groups.
Section~3 deals with homological and cohomological characterizations of
nilpotency in low degrees, and some global fusion criteria are studied in
Section 4. A criterion in terms of elementary abelian subgroups is described
in Section 5. Section 6 is then concerned with cohomological criteria in high
degrees, as well as Morava $K$-theory. Finally, Section 7 contains criteria
obtained by using Quillen categories.
\newline

\noindent {\bf Acknowledgements.} We would like to thank Nat\`alia
Castellana for helpful discussions. The authors would like to
thank the Max Plank Institute in Bonn and the Centre de Recerca Matem\`atica
in Barcelona where part of this work was done.

\section{Fusion systems and associated centric linking systems}
This first section is devoted to the basic definitions and
properties we will use from the beautiful theory of $p$-local finite
groups. Our main reference is the article \cite{MR1992826} by Broto,
Levi, and Oliver. The expert can happily skip this section.

Given two finite groups $P$, $Q$, let $\Hom (P,Q)$ denote the set of
group homomorphisms from $P$ to $Q$, and let $\Inj (P,Q)$ denote the
set of monomorphisms. If $P$ and $Q$ are subgroups of a larger group
$G$, then $\Hom _G(P,Q)\subseteq \Inj (P,Q)$ denotes the subset of
homomorphisms induced by conjugation by elements of $G$, and
Aut$_G(P)$ the group of automorphisms of $P$ induced by conjugation in $G$.

\begin{definition}
\label{def fusion}
{\rm A \emph{fusion system} $\Ff$ over a finite $p$-group $S$ is a
category whose objects are the subgroups of $S$, and whose morphism
sets $\Hom _{\Ff}(P,Q)$ satisfy the following conditions:

\begin{itemize}
\item[(a)] $\Hom _S(P,Q) \subseteq \Hom _{\Ff}(P,Q) \subseteq \Inj (P,Q)$ for all $ P,Q \leq S $.
\item[(b)] Every morphism in $\Ff$ factors as an isomorphism in $\Ff$ followed by an inclusion.
\end{itemize}
}
\end{definition}

If $\Ff$ is a fusion system over $S$ and $P,Q \leq S$, then we write
$\Hom _{\Ff}(P,Q)$ for the set of morphisms in $\Ff$ to emphasize
that all morphisms in the category $\Ff$ are homomorphisms. If $\Iso
_{\Ff}(P,Q)$ denotes the subset of isomorphisms in $\Ff$, we see
that $\Iso _{\Ff}(P,Q)= \Hom _{\Ff}(P,Q)$ if $|P|=|Q|$, and $\Iso
_{\Ff}(P,Q) = \emptyset $ otherwise. We also write $\aut _{\Ff}(P) =
\Iso _{\Ff}(P,P) $ and $\Out _{\Ff}(P) = \aut _{\Ff}(P)/\Inn (P)$.
Two subgroups $P,P' \leq S $ are called \emph{$\Ff$-conjugate} if
$\Iso _{\Ff}(P,P') \neq \emptyset $.

Here, and throughout the paper, we write $\Syl _p(G)$ for the set of
Sylow $p$-subgroups of $G$. Also, for any $ P \leq G $ and any $ g
\in N_G(P) $, $c_g \in \aut (P)$ denotes the automorphism $ c_g(x) =
gxg^{-1} $. The fusion systems we consider in this paper will all
be saturated in the following sense, \cite[Definition~1.2]{MR1992826}.

\begin{definition}
\label{def saturated}
{\rm Let $\Ff$ be a fusion system over a $p$-group $S$.

\begin{itemize}
\item A subgroup $ P \leq S $ is \emph{fully centralized} in $\Ff$ if $|C_S(P)| \geq |C_S(P')| $ for all $ P' \leq S $
that are $\Ff$-conjugate to $P$.
\item A subgroup $ P \leq S $ is \emph{fully normalized} in $\Ff$ if $|N_S(P)| \geq |N_S(P')| $ for all $ P' \leq S $
that are $\Ff$-conjugate to $P$.
\item $\Ff$ is a \emph{saturated fusion system} if the following two conditions hold:

\begin{itemize}
\item[(I)] Any $ P \leq S $ which is fully normalized in $\Ff$ is fully centralized in $\Ff$, and
$\aut _S(P) \in \Syl _p(\aut _{\Ff}(P))$.
\item[(II)] If $ P \leq S $ and $\phi \in \Hom _{\Ff}(P,S)$ are such that $\phi P $ is fully centralized, and if we set
\[
N_{\phi} = \{ g \in N_S(P) \mid \phi c_g \phi ^{-1} \in \text{Aut}_S(\phi P) \},
\]
then there is $\bar{\phi} \in \Hom _{\Ff}(N_{\phi},S)$ such that
$\bar{\phi} |_P = \phi$.

\end{itemize}

\end{itemize}}
\end{definition}

The motivating example for this definition is the fusion system of a
finite group $G$. For any $S \in \Syl _p(G)$, we let $\Ff _S(G)$ be
the fusion system over $S$ defined by setting $\Hom _{\Ff
_S(G)}(P,Q) = \Hom _G(P,Q)$ for all $ P,Q \leq S$.

\begin{proposition}{\rm (\cite[Proposition~1.3]{MR1992826})}
\label{prop groupfusion}
Let $G$ be a finite group, and let $S$ be a Sylow $p$-subgroup of
$G$. Then, the fusion system $\Ff _S(G) $ over $S$ is saturated.
Also, a subgroup $ P \leq S $ is fully centralized in $\Ff _S(G)$ if
and only if $C_S(P) \in \Syl _p(C_G(P))$, while $P$ is fully
normalized in $\Ff _S(G)$ if and only if $N_S(P) \in \Syl
_p(N_G(P))$.
\end{proposition}

In order to help motivate the next constructions, we recall some
definitions. If $G$ is a finite group and $p$ is a prime, then a
$p$-subgroup $ P \leq G $ is $p$-centric if $C_G(P) = Z(P) \times
C_G'(P) $, where $ C_G'(P)$ has order prime to $p$. For any $P,Q
\leq G $, let $N_G(P,Q)$ denote the \emph{transporter}, that is, the
set of all $ g \in G $ such that $ gPg^{-1} \leq Q $. For any $ S
\in \Syl _p(G)$, $\Ll ^c_S(G)$ denotes the category whose objects
are the $p$-centric subgroups of $S$, and where $\Mor _{\Ll
^c_S(G)}(P,Q)= N_G(P,Q)/C_G'(P)$. By comparison, $\Hom _G(P,Q) \cong
N_G(P,Q)/C_G(P) $. Hence there is a functor from $\Ll ^c_S(G) $ to
$\Ff _S(G)$ which is the inclusion on objects, and which sends the
morphism corresponding to $ g \in N_G(P,Q) $ to $ c_g \in \Hom
_G(P,Q)$.

\begin{definition}
\label{def centric}
{\rm Let $\Ff$ be any fusion system over a $p$-group $S$. A subgroup
$ P \leq S $ is \emph{$\Ff$-centric} if $P$ and all of its
$\Ff$-conjugates contain their $S$-centralizers. We denote by $\Ff
^c $ the full subcategory of $\Ff $ whose objects are the $\Ff
$-centric subgroups of $S$. }
\end{definition}

\begin{definition}
\label{def radical}
{\rm Let $\Ff$ be any fusion system over a finite $p$-group $S$. A
subgroup $ P \leq S $ is \emph{$\Ff$-radical} if $\Out _{\Ff}(P) $
is $p$-reduced, that is, if the maximal normal $p$-subgroup of $\Out
_{\Ff}(P) $ is $\{ 1 \} $. }
\end{definition}

\begin{theorem}{\rm (Alperin's fusion theorem,
\cite[Theorem~A.10]{MR1992826})} \label{AlperinFusion}
Let $\Ff$ be a saturated fusion system over a finite $p$-group $S$.
Then for each $ \phi \in \Iso _{\Ff} (P,P')$, there exist sequences
of subgroups of $S$
\[ P = P_0, P_1, \ldots , P_k = P' \text{ and } Q_1, Q_2, \ldots Q_k \]
and elements $ \phi _i \in \aut _{\Ff} (Q_i) $, such that

\begin{itemize}
\item[(a)] $Q_i $ is fully normalized in $\Ff$, $\Ff$-radical, and
$\Ff$-centric for each $i$ ;
\item[(b)] $P_{i-1}$, $P_i \leq Q_i $ and $ \phi _i(P_{i-1}) = P_i $ for
each $i$ ; and
\item[(c)] $ \phi = \phi _k \circ \phi _{k-1} \circ \ldots \circ \phi _1
$.
\end{itemize}
\end{theorem}

\begin{definition}
\label{def linkingsystem}
{\rm Let $\Ff $ be a fusion system over the $p$-group $S$. A
\emph{centric linking system} associated to $\Ff$ is a category
$\Ll$ whose objects are the $\Ff $-centric subgroups of $S$,
together with a functor
\[
\pi : \Ll \rightarrow \Ff ^c
\]
and ``distinguished" monomorphisms $ P \stackrel{\delta _P}{\to}
\aut _{\Ll}(P)$ for each $\Ff$-centric subgroup $ P \leq S $, which
satisfy the following conditions:

\begin{itemize}
\item[(A)] $\pi$ is the identity on objects and surjective on morphisms. More precisely, for each pair of objects $ P,Q \in \Ll $,
$Z(P)$ acts freely on $\Mor _{\Ll}(P,Q)$ by composition (upon
identifying $Z(P)$ with $\delta _P(Z(P)) \leq \aut _{\Ll}(P)$), and
$\pi$ induces a bijection
\[ \Mor _{\Ll}(P,Q)/Z(P) \stackrel{\cong}{\rightarrow} \Hom _{\Ff}(P,Q). \]
\item[(B)] For each $\Ff$-centric subgroup $ P \leq S $ and each $ g \in P $, $\pi $ sends $\delta _P(g) \in \aut _{\Ll}(P)$ to
$c_g \in \aut _{\Ff}(P) $.
\item[(C)] For each $ f \in \Mor _{\Ll}(P,Q) $ and each $g \in P$, the following square commutes in $\Ll$:
\[
\diagram P \rto^{f} \dto_{\delta _P(g)} & Q \dto^{\delta _Q(\pi
(f)(g))} \cr P \rto^f & Q.
\enddiagram
\]
\end{itemize}
}
\end{definition}

One easily checks that for any finite group $G$ and any $S \in \Syl
_p(G)$, $\Ll ^c_S(G)$ is a centric linking system associated to the
fusion system $\Ff _S(G)$. We are now ready to give the definition
of a $p$-local finite group, as in \cite[Definition~1.8]{MR1992826}.

\begin{definition}
{\rm A \emph{$p$-local finite group} is a triple $(S,\Ff,\Ll)$,
where $\Ff$ is a saturated fusion system over the $p$-group $S$ and
$\Ll$ is a centric linking system associated to $\Ff$. The
classifying space of the $p$-local finite group $(S,\Ff,\Ll)$ is the
space $|\Ll|^{\wedge}_p$. }
\end{definition}

Thus, for any finite group $G$ and any $ S \in \Syl _p(G)$, the
triple $(S,\Ff _S(G),\Ll ^c_S(G))$ is a $p$-local finite group. Its
classifying space $|\Ll ^c_S(P)|^{\wedge}_p $ is homotopy
equivalent to $BG^{\wedge}_p$ by \cite{MR1961340}. In general, the fact that $S$ itself
is an object of the centric linking system yields a ``canonical
inclusion" $Bi: BS \rightarrow |\Ll|$ of the Sylow subgroup $S$ into
the $p$-local finite group, which allows us to compare their mod $p$
cohomologies. We will denote by $ Bi \pcom $ the composition of $Bi$
with the $p$-completion map $ |\Ll| \rightarrow |\Ll| \pcom $.

\section{Nilpotent fusion systems and $p$-local finite groups}
In this section we start from the definition of nilpotency in terms of fusion data
and obtain different global characterizations, at the level of linking systems
and their mod $p$ cohomology. We start by giving an elementary proof of the fact that the fusion
system of a finite $p$-group has a unique associated centric linking system. This also follows
from \cite[Proposition~4.2]{MR2167090}, since a nilpotent
fusion system is constrained, or from Oliver's solution to the Martino-Priddy
conjecture in \cite{MR2203209} and \cite{MR2092063},
where it is shown that there is a unique centric linking system
associated to the fusion system of a finite group.

\begin{proposition}
\label{UniqueLinking}
If $S$ is a finite $p$-group, the fusion system $\Ff_S(S)$ has a unique
associated centric linking system up to isomorphism of categories.
\end{proposition}

\begin{proof}
Note that $\Ll_S(S)$ is a centric linking system associated to $\Ff_S(S)$. Let
$\Ll$ be another centric linking system associated to $\Ff_S(S)$. The categories
$\Ll_S(S)$ and $\Ll$ have the same objects, the $\Ff_S(S)$-centric subgroups of $S$.
Now consider the functor $ \delta: \Ll_S(S) \to \Ll $ constructed in \cite[Proposition~1.11]{MR1992826}.
Recall that $ \Mor_{\Ll_S(S)}(P,Q) = N_S(P,Q)$ for $\Ff_S(S)$-centric $ P $, $ Q \leq S $.
On morphisms, the maps $ \delta _{P,Q} : N_S(P,Q) \to \Mor_{\Ll}(P,Q) $ are injective
and $ \delta _{P,P}(g) = \delta _P (g)$ for $g \in P$. Therefore for all $\Ff_S(S)$-centric
$ P$, $Q \leq S $, there are commutative diagrams:
\[
\diagram
N_S(P,Q)/Z(P) \rto \dto & \Mor_{\Ll}(P,Q)/Z(P) \dto \cr
\Hom _S(P,Q) \rto & \Hom _{\Ff_S(S)}(P,Q), \cr
\enddiagram
\]
where the two vertical maps are isomorphisms by the definition of centric linking system and
the lower horizontal map is the identity. Hence the upper horizontal map must be an isomorphism
as well. Since the action of $Z(P)$ is free and compatible with the maps $ \delta_{P,Q} $, this
implies that $ \delta_{P,Q}$ is an isomorphism and therefore the categories $\Ll$ and $\Ll_S(S)$
are isomorphic.
\end{proof}

Following Kessar and Linckelmann's definition for fusion systems in
\cite{MR2379788} and the characterization of $p$-nilpotence for
honest groups in terms of fusion stated in the introduction, we define a
$p$-local finite group to be nilpotent when its fusion system is that
of a $p$-group.

\begin{definition}
\label{def nilpotent}
{\rm A $p$-local finite group $(S, \Ff, \Ll)$ is \emph{nilpotent} if
$\Ff = \Ff _S(S)$.}
\end{definition}

Even though this definition seems to be stated at the level of fusion systems only,
it says that the $p$-local finite group is an honest $p$-group. As it is sometimes more
convenient to check a topological condition (using the centric linking system)
than to stay at the purely algebraic level of the fusion data, we propose the following
characterization.

\begin{proposition}
\label{prop Martino-Priddy}
A $p$-local finite group $(S, \Ff, \Ll)$ is nilpotent if and only if $|\Ll| \pcom \simeq BS$.
\end{proposition}

\begin{proof}
If $ \Ff = \Ff_S(S)$, then $\Ll$ is isomorphic to $\Ll_S(S)$ by Proposition \ref{UniqueLinking}
and in particular $ |\Ll| \pcom \simeq |\Ll_S(S)| \pcom \simeq BS $.

On the other hand, if $ |\Ll| \pcom \simeq BS $, let $ f : BS \to | \Ll | \pcom $ be a homotopy
equivalence. By \cite[Theorem~4.4(a)]{MR1992826}, $ f $ is homotopic to $ Bi \pcom \circ B\rho$
for some $ \rho \in \Hom(S,S)$. Therefore, $ Bi \pcom \simeq f \circ B(\rho^{-1}) $ is also a homotopy equivalence.

Now \cite[Theorem~7.4]{MR1992826} says that if $ |\Ll| \pcom \simeq |\Ll_S(S)| \pcom $, then the $p$-local
finite groups $(S,\Ff,\Ll)$ and $(S,\Ff_S(S),\Ll_S(S))$ are isomorphic, in the sense that there are isomorphisms
of groups and categories $ \alpha: S \to S $, $  \alpha_{\Ff} : \Ff_S(S) \to \Ff $ and $  \alpha_{\Ll} : \Ll_S(S) \to \Ll $
such that $ \alpha _{\Ff} (P) = \alpha _{\Ll}(P) = \alpha (P) $ and such that they commute with the projections $ \Ll \to \Ff ^c $
and the structure maps $ \delta_Q : Q \to \aut_{\Ll}(Q) $ for all $ Q \leq S$. The proof of the theorem shows that we can
choose $\alpha : S \to S $ to be any isomorphism that makes the following diagram commute up to homotopy:
\[
\diagram
BS \rto^{id} \dto_{B\alpha} & BS \dto^{Bi \pcom} \cr
BS \rto^{Bi \pcom} & | \Ll | \pcom.
\enddiagram
\]

We can choose the identity map and in particular we have an isomorphism of categories $ \alpha_{\Ff} : \Ff_S(S) \to \Ff $ such that $ \alpha_{\Ff}(P) = P $
for all $ P \leq S $. Therefore $ \Hom_{\Ff}(P,Q) = \Hom_{\Ff_S(S)}(P,Q)$ so $\Ff = \Ff_S(S)$.
\end{proof}

Our first cohomological criterion is an elementary reformulation of the definition.

\begin{proposition}
\label{prop cohomology}
A $p$-local finite group $(S, \Ff, \Ll)$ is nilpotent if and only if
the canonical inclusion induces an isomorphism $H^*(|\Ll|, \F_p) \to
H^*(BS, \F_p)$.
\end{proposition}

\begin{proof}
By Proposition \ref{prop Martino-Priddy}, if $(S,\Ff,\Ll)$ is nilpotent, then
$ Bi \pcom : BS \to| \Ll | \pcom $ is a homotopy equivalence and therefore
the induced map $ H_*(BS;\F_p) \to H_*(|\Ll|, \F_p) $ is an isomorphism.
By \cite[Theorem~2.48]{MR2378355}, the induced map in mod $p$ cohomology
is also an isomorphism.

By \cite[Proposition~5.2]{MR1992826} and \cite[Theorem~5.8]{MR1992826}, the ring
$H^*(|\Ll|, \F_p) $ is Noetherian, and in particular the groups $H^n(|\Ll|, \F_p)
\cong H_n(|\Ll|, \F_p)$ are finitely generated $\F_p$-modules. Again by \cite[Theorem~2.48]{MR2378355},
an isomorphism in mod $p$ cohomology implies an isomorphism in mod $p$ homology,
which is equivalent to a homotopy equivalence between their $p$-completions.
\end{proof}

\begin{proposition}
\label{retract}
Let $(S,\Ff,\Ll)$ be a $p$-local finite group and let $Bi: BS
\rightarrow |\Ll|$ be the standard inclusion. Then $(S,\Ff,\Ll)$
is nilpotent if and only if the map $Bi^\wedge_p$ has a retraction
$r: |\Ll|^\wedge_p \rightarrow BS$.
\end{proposition}

\begin{proof}
The existence of the retraction implies that the cohomology of the
Sylow subgroup is a retract of the cohomology of $|\Ll|^\wedge_p$.
But the mod $p$ cohomology of $|\Ll|$ can be computed by stable
elements by \cite[Theorem~5.8]{MR1992826}, and moreover it is a subalgebra of $H^*(S; \F_p)$. Therefore $BS \rightarrow
|\Ll|^\wedge_p$ must be an equivalence.

On the other hand, if $ \Ff = \Ff_S(S) $, then there is a unique
centric linking system associated to $\Ff$ up to isomorphism by Proposition
\ref{UniqueLinking}. Given an isomorphism of categories $ \Ll \to \Ll_S(S)$ compatible with the
projection maps to $ \Ff = \Ff_S(S)$ and the structure maps $ \delta_P$,
the induced map on classifying spaces gives such a retraction.
\end{proof}

\begin{remark}
\label{stableretract}
{\rm The result \cite[Proposition~5.5]{MR1992826} on which the proof
of the stable element formula relies is stronger than the use we
make of it in the previous proof. It actually tells us that the
suspension spectrum of $|\Ll|^\wedge_p$ appears as a stable summand
in the suspension spectrum $\Sigma^\infty BS$. In particular
$Bi^\wedge_p$ induces always an epimorphism in homology and a
monomorphism in cohomology with arbitrary coefficients.}
\end{remark}

\section{Cohomological characterizations in low degrees}
\label{sectionlow}

We show in this section that a nilpotent $p$-local finite group can
be recognized by looking at (co)homological information in low
degree. For finite groups the criteria we obtain are those of Tate
\cite{MR0160822}, Stammbach \cite{MR0473039}, and some variations.
Proposition~\ref{retract} is useful in view of the following result,
which establishes a ``semi-localization property" for certain maps
$BS \rightarrow X$.

\begin{proposition}
\label{weaklyinitial}
Let $S$ be a $p$-group and $f: BS \rightarrow X$ be a map inducing
an epimorphism $H^1(X; \F_p) \twoheadrightarrow H^1(BS; \F_p)$ and
a monomorphism $H^2(X; \F_p) \hookrightarrow H^2(BS; \F_p)$. Then
any map $BS \rightarrow BP$ to the classifying space of a
$p$-group $P$ factors through $f$ up to homotopy. In particular
$BS$ is a retract of $X$.
\end{proposition}

\begin{proof}
We proceed by induction on the order of $P$. If $P \cong C_p$, the
cyclic group of order $p$, the conclusion is a direct consequence of
the surjectivity of $H^1(f; \F_p)$. Let us thus consider a $p$-group
$P$ of order $p^n$ with $n \geq 2$ and a group homomorphism $\phi: S
\rightarrow P$.

The center of a $p$-group is never trivial, therefore there is a
central extension $C_p \hookrightarrow P
\stackrel{\pi}{\twoheadrightarrow} Q$ and by induction there
exists a map $g: X \rightarrow BQ$ such that the following square
commutes up to homotopy
\[
\diagram BS \rto^{B\phi} \dto_f & BP \dto^{B\pi} \cr
X \rto^g & BQ.
\enddiagram
\]
Now we have to lift $g$ to a map $h: X \rightarrow BP$. The central
extension gives rise to a fibration $BP \rightarrow BQ
\stackrel{k}{\longrightarrow} K(\Z/p, 2)$. The composite map $k
\circ B\pi \circ B\phi \simeq k \circ g \circ f$ is null-homotopic,
and hence so is $k \circ g: X \rightarrow K(\Z/p, 2)$ because of the
injectivity of $H^2(f; \F_p)$. Therefore there exists a map $h: X
\rightarrow BP$ such that $B\pi \circ h \simeq g$.

In general, $h \circ f \not\simeq B\phi$ because they can differ by
the action of the central $C_p$. This means,
\cite[IX.4.1]{MR51:1825}, that there exists a map $\alpha: BS
\rightarrow BC_p$ such that the composite map
$$
BS \stackrel{\Delta}{\longrightarrow} BS \times BS \xrightarrow{(h
\circ f) \times \alpha} BP \times BC_p \rightarrow BP
$$
is equal to $B\phi$. By the induction hypothesis again, there
exists a map $\beta: X \rightarrow BC_p$ such that $\beta \circ h
\simeq \alpha$. The composite map
$$
X \stackrel{\Delta}{\longrightarrow} X \times X \xrightarrow{h
\times \beta} BP \times BC_p \rightarrow BP
$$
is the map we are seeking.
\end{proof}

We emphasize that the factorization through $f$ in the statement of
the previous proposition is not unique. So $f$ is only ``weakly
initial" for maps to classifying spaces of $p$-groups, or, in other
words, $BP$ is not quite local with respect to~$f$ (in the sense of e.g. \cite{MR1392221}).
This property is the key ingredient in the proof of the $p$-local version of
Stammbach's criterion \cite{MR0473039} about the second (co)homology
group and the Huppert-Thompson-Tate criterion \cite{MR0160822} about
the first one. Our proof goes along the same lines as in Stammbach's
note.

\begin{theorem}
\label{lowdegree}
Let $(S,\Ff,\Ll)$ be a $p$-local finite group and let $Bi: BS
\rightarrow |\Ll|$ be the standard inclusion. Then $(S,\Ff,\Ll)$
is nilpotent if and only if one of the following four conditions
is satisfied:

\begin{itemize}
\item[(1)] $Bi^*: H^1(|\Ll|; \F_p) \rightarrow H^1(BS; \F_p)$
is an isomorphism.

\item[(2)] $Bi^*: H^2(|\Ll|; \F_p) \rightarrow H^2(BS; \F_p)$ is an isomorphism.

\item[(3)] $Bi_*: H_1(BS; \F_p) \rightarrow H_1(|\Ll|; \F_p)$
is an isomorphism.

\item[(4)] $Bi_*: H_2(BS; \F_p) \rightarrow H_2(|\Ll|; \F_p)$ is an isomorphism.
\end{itemize}

\end{theorem}

\begin{proof}
The universal coefficient formula clearly implies the equivalence
of the homological conditions and the cohomological ones. We
follow Stammbach's strategy \cite{MR0473039} and assume that (4)
holds. Consider the universal coefficients exact sequences:
\[
\diagram
0 \rto & H_2(S; \Z) \otimes \Z/p \rto \dto^{\alpha} & H_2(S; \F_p) \rto \dto^{\cong} & \tor(H_1(S; \Z); \Z/p) \dto^{\beta} \rto & 0 \\
0 \rto & H_2(|\Ll|^\wedge_p; \Z) \otimes \Z/p \rto & H_2(|\Ll|^\wedge_p; \F_p) \rto & \tor(H_1(|\Ll|^\wedge_p; \Z); \Z/p) \rto & 0. \\
\enddiagram
\]
We see that $\alpha$ must be a monomorphism and $\beta$ an
epimorphism. By Remark~\ref{stableretract} we know that $\alpha$
is an epimorphism, so that $\beta$ is actually an isomorphism.

Now the short exact sequence $K= \Ker(H_1(Bi; \Z)) \hookrightarrow
H_1(S; \Z) \twoheadrightarrow H_1(|\Ll|^\wedge_p; \Z)$ induces an
exact sequence
$$
0 \rightarrow \tor(K; \Z/p) \rightarrow \tor(H_1(S; \Z); \Z/p)
\xrightarrow{\beta} \tor(H_1(|\Ll|^\wedge_p; \Z); \Z/p),
$$
so that $\tor(K; \Z/p)=0$. Since the homology of $S$ is $p$-torsion,
$K$ must be trivial. Therefore $H_1(Bi; \Z)$ is an isomorphism, and
in particular (3) holds.

It remains to prove that condition~(3) -- or equivalently
condition~(1) -- implies that $(S,\Ff,\Ll)$ is nilpotent. We assume
therefore that $H^1(Bi; \F_p)$ is an isomorphism. Since $H^2(Bi;
\F_p)$ is always a monomorphism, Proposition~\ref{weaklyinitial}
applies and we deduce that $BS$ is a retract of $|\Ll|^\wedge_p$.
This means by Proposition~\ref{retract} that the $p$-local finite
group is nilpotent.
\end{proof}

Using the appropriate universal coefficients theorems for the
$\Z^\wedge_p$-module $\F_p$, one can easily replace the $\F_p$
coefficients by the $p$-adic integers.

\begin{corollary}
\label{pcompletedlowdegree}
Let $(S,\Ff,\Ll)$ be a $p$-local finite group and let $Bi: BS
\rightarrow |\Ll|$ be the standard inclusion. Then $(S,\Ff,\Ll)$
is nilpotent if and only if one of the following two conditions is
satisfied:

\begin{itemize}
\item[(1)] $Bi_*: H_1(|\Ll|; \Z^\wedge_p) \rightarrow H_1(BS; \Z^\wedge_p)$
is an isomorphism.

\item[(2)] $Bi^*: H^2(|\Ll|; \Z^\wedge_p) \rightarrow H^2(BS; \Z^\wedge_p)$
is an isomorphism. \hfill{\qed}
\end{itemize}

\end{corollary}

\section{Global fusion criteria}
\label{sec global}
This section is devoted to criteria which allow to recognize a
nilpotent $p$-local finite group from its global fusion
characteristics. We first need some facts about the fundamental
group of a $p$-local finite group. The following definition is due
to Puig \cite{Puig}. The notation $O^p(-)$ stands for the subgroup
generated by elements of order prime to~$p$.

\begin{definition}
\label{focal}
{\rm For any saturated fusion system $\Ff$ over a $p$-group $S$,
the {\em hyperfocal subgroup} $Hyp (\Ff)$ is the normal subgroup
$\langle g^{-1}\alpha(g)\, |\, g\in P\leq S, \alpha\in
O^p(\aut_{\Ff}(P))\rangle$. The {\em focal subgroup} $Foc (\Ff)$ is
the normal subgroup $\langle g^{-1}\alpha(g)\, |\, g\in P\leq S,
\alpha\in \aut_{\Ff}(P)\rangle$.}
\end{definition}

\begin{theorem}{\rm (\cite[Theorem~B]{MR2302515})}
\label{bcglo-thB}
Let $(S,\Ff,\Ll)$ be a $p$-local finite group. Then
$\pi_1(|\Ll|\pcom) \cong S/Hyp (\Ff)$. \hfill{\qed}
\end{theorem}

We have seen in the previous section that the group
$H_1(|\Ll|\pcom;\Z\pcom) \cong \pi_1(|\Ll|\pcom)_{ab}$ plays an
important role in establishing the nilpotence of a given $p$-local
finite group. The computation in the next proposition (compare with
\cite[Theorem~7.3.4]{MR569209}) explains the relation between the
focal and the hyperfocal subgroups.

\begin{proposition}
\label{prop H1Foc}
Let $(S,\Ff,\Ll)$ be a $p$-local finite group. Then
$H_1(|\Ll|\pcom;\Z\pcom) \cong S/Foc (\Ff)$.
\end{proposition}

\begin{proof}
The group $Foc (\Ff)$ is generated by $Hyp (\Ff)$ and elements
$g^{-1}\alpha(g)$ where $\alpha\in \aut_{\Ff}(P)\setminus
O^p(\aut_{\Ff}(P))$ for some $P \leq S$. By Alperin's fusion theorem, Theorem~\ref{AlperinFusion},
it is enough to consider those subgroups $P$ that are fully normalized in $\Ff$.

Assume then that $P$ is fully normalized in $\Ff$. Since $\aut_S(P) \in \Syl_p(\aut_{\Ff}(P))$,
any elements of $p$-power order in $\aut_{\Ff}(P)$ can be conjugated to a
conjugation in $S$ via an element of $O^p(\aut_{\Ff}(P))$. Therefore
$Foc (\Ff)$ is generated by $Hyp (\Ff)$ and commutators in $[S,S]$.
In view of Theorem~\ref{bcglo-thB}, we conclude that the abelianization
of $\pi_1(|\Ll|\pcom) \cong S/Hyp (\Ff)$ is $S/Foc (\Ff)$.
\end{proof}

Recall the following definitions from \cite[Definition~A.3]{MR1992826}.

\begin{definition}
\label{NormalizerCentralizer}
{\rm The \emph{normalizer fusion system} $N_{\Ff}(Q)$ of a subgroup $Q \leq S$ in $\Ff$ is the
fusion system defined over $N_S(Q)$ whose morphisms are given by:
\[ \Hom_{N_{\Ff}(Q)}(P,P') = \{ \phi \in \Hom _{\Ff}(P,P') \mid \exists \psi \in \Hom _{\Ff}(PQ,P'Q),\hspace{1mm} \psi|_P = \phi,\hspace{1mm} \psi(Q) \leq Q \}. \]
The \emph{centralizer fusion system} $C_{\Ff}(Q)$ of a subgroup $Q \leq S$ in $\Ff$ is the
fusion system defined over $C_S(Q)$ whose morphisms are given by:
\[ \Hom_{C_{\Ff}(Q)}(P,P') = \{ \phi \in \Hom _{\Ff}(P,P') \mid \exists \psi \in \Hom _{\Ff}(PQ,P'Q),\hspace{1mm} \psi|_P = \phi,\hspace{1mm} \psi|_Q = \id \}. \]
}
\end{definition}

These definitions extend to centric linking systems, see \cite[Definition~2.4]{MR1992826}
and \cite[Definition~6.1]{MR1992826}.

\begin{definition}{\rm (\cite[Definition~1.5]{MR2167090})}
\label{WeaklyClosed}
{\rm A subgroup of $S$ is $\Ff$-{\em weakly closed} if no other
subgroup is $\Ff$-conjugate to it.}
\end{definition}

\begin{lemma}
\label{center-wc}
Let $(S,\Ff,\Ll)$ be a $p$-local finite group and let $V\leq Z(S)$
be $\Ff$-weakly closed. Then $V$ is normal in $(S,\Ff,\Ll)$, that
is, $(S,\Ff,\Ll) = (N_S(V),N_{\Ff}(V), N_{\Ll}(V))$.
\end{lemma}

\begin{proof}
By Alperin's fusion theorem for saturated fusion systems,
Theorem~\ref{AlperinFusion}, every morphism in $\Ff$ is the
composition of automorphisms of $\Ff$-centric subgroups $P_i$. More
precisely, every morphism $\alpha \in \Mor_{\Ff}(P,Q)$ is the
composition of appropriate restrictions of certain $\alpha_i \in
\aut_{\Ff}(P_i)$. Notice that $V\leq P_i$ since $P_i$ is centric and
$\alpha_i(V)\leq V$ since $V$ is weakly closed in $(S,\Ff,\Ll)$.
Thus $\alpha$ can be extended to a morphism from $VP$ to $VQ$.
\end{proof}

A straightforward consequence is the following $p$-local version of
Gr\"un's theorem \cite[Theorem~7.5.2]{MR569209}.

\begin{proposition}
\label{WeaklyFocal}
Let $(S,\Ff,\Ll)$ be a $p$-local finite group and let $A\leq Z(S)$
be weakly closed in $\Ff$. Then $Foc(\Ff) = Foc (N_{\Ff}(A))$.
\end{proposition}

\begin{proof}
By Lemma \ref{center-wc}, $A$ is normal in $\Ff$, so the fusion
systems $\Ff$ and $N_{\Ff}(A)$ actually coincide.
\end{proof}

We are now ready to give our ``global fusion criteria". At the level
of honest groups, the first criterion is due to Huppert and is
proven in \cite[Satz~IV.4.9]{MR0224703} by means of the abelian
transfer. We propose also a variation. The third one is called
Frobenius $p$-nilpotency criterion in the literature, see for
example \cite[10.3.2]{MR648604}, and the fourth criterion is a stronger
form of this criterion.

\begin{theorem}
\label{global}
A $p$-local finite group $(S,\Ff,\Ll)$ is nilpotent if and only if one
of the following four conditions is satisfied:

\begin{itemize}
\item[(1)] Two elements $a,b\in S$ are $\Ff$-conjugate if and
only if they are $S$-conjugate.

\item[(2)] Two $n$-tuples of commuting
elements of $S$ are $\Ff$-conjugate if and only if they are
$S$-conjugate, $n>1$.

\item[(3)] For every subgroup $P\leq S$, $\aut_{\Ff}(P)$ is a $p$-group.

\item[(4)] For every $\Ff$-centric subgroup $P \leq S$, $\aut_{\Ff}(P)$ is a $p$-group.
\end{itemize}

\end{theorem}

\begin{proof}
It is clear that all four conditions hold for a nilpotent $p$-local
finite group. Let us thus assume that condition~(2) holds. In
particular for any two elements $a$ and $b$ in $S$ which are
$\Ff$-conjugate, the $n$-tuples $(a, \, \dots, a)$ and $(b, \,
\dots, b)$ are also $\Ff$-conjugate. Therefore they must be
$S$-conjugate, which proves that $a$ and $b$ are $S$-conjugate. This
shows that (2) implies~(1).

Assume now that condition (1) holds, so that the image of any element
under an $\Ff$-automorphism in $S$ is a conjugate of that element by
some element in $S$, that is both the focal and hyperfocal subgroups
are contained in the commutator subgroup $[S,S]$. Therefore, according to Proposition~\ref{prop
H1Foc}, the identity of $S/[S, S]$ factors through the map
$H_1(Bi;\Z\pcom): S/[S, S] \rightarrow H_1(|\Ll|\pcom;\Z\pcom)\cong
S/Foc(\Ff)$. We have seen in Remark~\ref{stableretract} that this
map is always an epimorphism. It is hence an isomorphism and we have
proven that the $p$-adic characterization~(1) in
Corollary~\ref{pcompletedlowdegree} holds.

It is obvious that (3) implies~(4). Finally, if condition (4) holds, and $ P \lneq S$
is an $\Ff$-centric subgroup of $S$, then $ P \lneq N_S(P)$ and so $ \Inn(P) \lneq \aut _S(P)$.
In particular, $ \Out_S(P) $ is a non-trivial subgroup of $ \Out_{\Ff}(P) $ and so $P$
is not $\Ff$-radical. Therefore the only $\Ff$-radical $\Ff$-centric subgroup is $S$. By Alperin's fusion theorem,
Theorem \ref{AlperinFusion}, every morphism in $\Ff$ is the restriction of an $\Ff$-automorphism of $S$. On the other hand, since $S$ is fully normalized
in $\Ff$, $ \Inn(S) \in \Syl _p(\aut_{\Ff}(S)) $, which implies $ \aut_{\Ff}(S) = \Inn(S)$. Hence every
morphism in $\Ff$ is induced by $S$-conjugation, that is, $ \Ff = \Ff_S(S)$.
\end{proof}

The following corollary is a very simple criterion to determine nilpotency of $p$-local finite groups
with an abelian $p$-Sylow.
\begin{corollary}
\label{abelian}
Let $(S,\Ff,\Ll)$ be a $p$-local finite group with $S$ abelian.
Then $(S,\Ff,\Ll)$ is nilpotent if and only if one of the following two conditions
is satisfied:

\begin{itemize}
\item[(1)] Two elements $a,b\in S$ are $\Ff$-conjugate if and
only if they are equal.

\item[(2)] $\aut_{\Ff}(S)$ is the trivial group. \hfill{\qed}

\end{itemize}
\end{corollary}

Theorem~\ref{global} allows us to obtain for example a $p$-local version of
Huppert's result \cite[Satz~III.12.1]{MR0224703}.

\begin{proposition}
\label{p_is_central}
Let $(S,\Ff,\Ll)$ be a $p$-local finite group such that every
element of order $p^n>2$ in $S$ is central in $\Ff$. Then
$(S,\Ff,\Ll)$ is nilpotent.
\end{proposition}

\begin{proof}
We use the Frobenius type characterization (4) in
Theorem~\ref{global} of nilpotent $p$-local finite groups. Let $K$
be the subgroup generated by all the elements of order $p^n$ in $S$.
This abelian subgroup has exponent smaller than or equal to $p^n$ and it is maximal
with respect to these two conditions. Let $P\leq S$ be an $\Ff$-centric subgroup,
and let $\alpha\in\aut_{\Ff}(P)$. Since every element of order $p^n>2$ is
central in $\Ff$, $\alpha$ acts trivially on every element of order
$p^n$ in $P$. In particular, it acts trivially on $K$
and, by \cite{MR0190238},
$\alpha$ is a $p$-element, i.e., $\aut_{\Ff}(P)$ is a $p$-group.
\end{proof}

\section{Quillen's first criterion}
\label{sec Quillen}
The criterion we offer in this section is an extension of Quillen's
\cite[Theorem~1.5]{MR0318339}, which is of course done in the
setting of honest groups.  The following lemma was originally
\cite[Proposition~4.1]{MR0318339}, in which Quillen only considers
elementary abelian subgroups of the Sylow $p$-subgroup. Here the
exponent is arbitrary.

\begin{lemma}
\label{quillen 4.1}
Let $(S,\Ff,\Ll)$ be a $p$-local finite group and let $A$ be a
subgroup of $S$ maximal subject to being normal abelian and of
exponent $p^n>2$. Then $A$ is also maximal subject to being
abelian and of exponent $p^n$ in  $(S,\Ff,\Ll)$, that is, any
$\Ff$-conjugate of $A$ is maximal subject to being abelian and
of exponent $p^n$ in $S$.
\end{lemma}

\begin{proof}
Assume that there exist abelian subgroups $A'\leq W'\leq S$  of
exponent $p^n>2$ such that $A$ and $A'$ are $\Ff$-conjugate by the
morphism $\varphi:A'\rightarrow A$. Since $W'\leq C_S(A')\leq
N_\varphi$, then $W'$ is $\Ff$-conjugate to another abelian subgroup
$W$ that contains $A$. But according to \cite{MR0167528} (see also
\cite[Satz III.12.1]{MR0224703}) 
$A$ is maximal in $S$ subject to being abelian and of exponent
$p^n$, hence $W=A$, $W'=A'$ and the result follows.
\end{proof}

The following result provides a way to show that certain subgroups
are central in a $p$-local finite group.

\begin{proposition}
\label{center-aut_es_p-gp}
Let $(S,\Ff,\Ll)$ be a $p$-local finite group and let $V\leq Z(S)$
be an $\Ff$-weakly closed subgroup such that $\aut_{\Ff}(V)$ is a
$p$-group. Then $V$ is central in $(S,\Ff,\Ll)$, that is, $(S,\Ff,\Ll) = (C_S(V),C_{\Ff}(V),C_{\Ll}(V))$.
\end{proposition}

\begin{proof}
By Lemma~\ref{center-wc} we know that
$(S,\Ff,\Ll)=(S,N_{\Ff}(V),N_{\Ll}(V))$. Note that $V$ is $\Ff$-weakly
closed, and so it is fully normalized and fully centralized. Therefore
$\aut_S(V) \in \Syl_p(\aut_{\Ff}(V)) $. Since $\aut_{\Ff}(V)$ is a $p$-group,
it must be $\aut_S(V) = N_S(V)/C_S(V) = S/S = \{1 \}$. We will show that the
two $p$-local finite groups $(S,C_{\Ff}(V),C_{\Ll}(V))$ and $(S,N_{\Ff}(V),N_{\Ll}(V))$
are equal.

Clearly $\Hom_{C_{\Ff}(V)}(P,Q) \subseteq \Hom_{N_{\Ff}(V)}(P,Q) $ for
all $ P $, $Q \leq S $. Let $\phi:P \to Q$ be a morphism in $N_{\Ff}(V)$, then there is $ \psi : PV \to QV $ in
$\Ff$ such that $\psi |_P = \phi $ and $ \psi (V) \leq V $. Since $\psi |_V $ is a morphism
in $\Ff$, it must be the identity, hence $ \phi $ is a morphism in $C_{\Ff}(V)$.
Now, if $P$ and $Q$ are $\Ff$-centric, then the morphisms for $C_{\Ll}(V)$ and $N_{\Ll}(V)$ as defined in \cite[Definition~2.4]{MR1992826}
and \cite[Definition~6.1]{MR1992826} are given by $\Mor _{C_{\Ll}(V)}(P,Q) = \pi ^{-1}(\Hom_{C_{\Ff}(V)}(P,Q))$
and $\Mor _{N_{\Ll}(V)}(P,Q) = \pi ^{-1}(\Hom_{N_{\Ff}(V)}(P,Q))$, where $ \pi : \Ll \to \Ff ^c $ is the projection functor
from \ref{def linkingsystem}. This shows that $ C_{\Ll}(V) = N_{\Ll}(V) $.
\end{proof}

Before stating Quillen's characterization, we start with a special
case.

\begin{lemma}
\label{blackburn}
Let $(S,\Ff,\Ll)$ be a $p$-local finite group at an odd prime $p$.
Assume that there exists a maximal elementary abelian normal
subgroup $V \triangleleft S$ which is $\Ff$-normal and such that
$\aut_{\Ff}(V)$ is a $p$-group. Then $(S,\Ff,\Ll)$ is nilpotent.
\end{lemma}

\begin{proof}
We show that for every $P\leq S$, $\aut_{\Ff}(P)$ is a $p$-group,
thereby verifying condition~(3) in Theorem~\ref{global}. Let
$\alpha\in \aut_{\Ff}(P)$ be an element of order $q$ prime to~$p$. Then
it is the restriction of some element $\tilde{\alpha}\in \aut_{\Ff}(PV)$. Let the order
of $\tilde{\alpha} $ be $ ar $ where $ a $ is a multiple of $q$ that is prime to $r$ and~$p$.
In particular, $r$ is a prime to $q$ and so there is a positive integer $l$ prime to $q$ such
that $ lr $ is congruent to $1$ mod $q$. The order of $ \tilde{\alpha}^{lr} $ divides $a$, in
particular it is prime to $p$, and the restriction to $P$ is $\alpha$ again. Therefore, we
may assume that $\tilde{\alpha}$ has order prime to $p$.

Now, $\tilde{\alpha}|_V$ is an element in the $p$-group $\aut_{\Ff}(V)$,
it must be the identity. As $p>2$, Lemma~\ref{quillen 4.1} applies
and so $V$ is a maximal elementary abelian subgroup of $S$ and hence
of $PV$. We deduce from \cite{MR0190238}
that the order of $\tilde{\alpha}$
is a power of~$p$, hence $\tilde{\alpha}=1$ and therefore
$\alpha=1$.
\end{proof}

\begin{theorem}
\label{quillen1}
Let $(S,\Ff,\Ll)$ be a $p$-local finite group at an odd prime $p$.
Then $(S,\Ff,\Ll)$ is nilpotent if and only if every elementary
abelian normal subgroup $V\triangleleft S$ is $\Ff$-weakly closed
and $\aut_{\Ff}(V)$ is a $p$-group.
\end{theorem}

\begin{proof}
This condition is necessary in view of Theorem~\ref{global}. Let
thus $(S,\Ff,\Ll)$ be a $p$-local finite group such that every
elementary abelian normal subgroup $V\triangleleft S$ is
$\Ff$-weakly closed and $\aut_{\Ff}(V)$ is a $p$-group. We show that
it must then be nilpotent.

Consider first the elementary abelian normal subgroup $V_0={}_p
Z(S)$. It is weakly closed in $(S,\Ff,\Ll)$ by hypothesis, thus
normal in $\Ff$ by Lemma~\ref{center-wc}. According to
Proposition~\ref{center-aut_es_p-gp}, $(S,\Ff,\Ll)=(S,C_{\Ff}(V_0),C_{\Ll}(V_0))$,
hence $V_0$ is central in $(S,\Ff,\Ll)$.

Let $V\leq S$ be a maximal normal elementary abelian subgroup of $S$
containing $V_0$. If $V=V_0$, then it is $\Ff$-normal
and by Lemma~\ref{blackburn}, $(S,\Ff,\Ll)$ is nilpotent. Let us thus
assume that $V_0$ is strictly contained in $V$.

Consider the quotient $p$-local finite group
$(S_1,\Ff_1,\Ll_1)=(S/V_0,\Ff/V_0,\Ll/V_0)$, as in
\cite[Lemma~5.6]{MR1992826}. From the fibration $$ BV_0\rightarrow
|\Ll|\pcom \rightarrow |\Ll_1|\pcom, $$ we see that $(S,\Ff,\Ll)$ is
nilpotent if and only if $(S_1, \Ff_1, \Ll_1)$ is so. Let $V_1$ be the
only subgroup satisfying $V_0< V_1 \leq V$ and $V_1/V_0=V/V_0\cap
Z(S_1)$.

By construction $V_1$ is elementary abelian and normal in $S$. It is
also strictly larger than $V_0$ because $V/V_0$ is a non-trivial
normal subgroup of $S/V_0$ and therefore intersects non-trivially
its center, \cite[Theorem~2.6.4]{MR569209}. By the hypothesis, $V_1$ is
weakly closed in $\Ff$, so $V_1/V_0$ is weakly closed in $\Ff_1$.
By Lemma \ref{center-wc}, $V_1/V_0$ is normal in $(S_1,\Ff_1,\Ll_1)$.

Let $ f:P/V_0 \to Q/V_0$ be a map in $\Ff_1$. Since $V_1/V_0$ is
normal in $\Ff_1$, $f$ is the restriction of a map $ g : PV_1/V_0 \to QV_1/V_0$
in $\Ff_1 = \Ff/V_0$. This map must be induced by a morphism $ h: PV_1 \to QV_1 $
in $\Ff$, which must satisfy $h(P) \leq Q $ and $h(V_1)=V_1$, since both $P$ and $Q$ contain
$V_0$ and $V_1$ is $\Ff$-weakly closed. This shows that $ \Ff_1 = N_{\Ff}(V_1)/V_0 $,
and so there is a fibration
\[ BV_0 \rightarrow|N_{\Ll}(V_1)|\pcom \rightarrow |\Ll_1|\pcom , \]
from where we deduce that $(S_1,\Ff_1,\Ll_1)$ is nilpotent if and only if
$(S,\Ff^1,\Ll^1)=(S,N_{\Ff}(V_1),N_{\Ll}(V_1))$ is nilpotent. The fusion system
$\Ff^1$ is a subcategory of $\Ff$ with the same objects, so it also verifies
that every elementary abelian normal subgroup $W \triangleleft S$ is $\Ff^1$-weakly
closed and $\aut_{\Ff^1}(W)$ is a $p$-group. Moreover, $V_1$ is normal in $(S,\Ff^1,\Ll^1)$.

If $V_1\neq V$ we can iterate the process by defining $V_2$ to be the
only subgroup satisfying $V_1<V_2\leq
V$ and $V_2/V_1=V/V_1\cap Z(S/V_1)$. The $p$-local
finite group $(S,\Ff^1,\Ll^1)$ is nilpotent if and only if a new
$p$-local finite group $(S,\Ff^2,\Ll^2)$, which verifies that
every elementary abelian normal subgroup $W\triangleleft S$ is $\Ff^2$-weakly closed
and $\aut_{\Ff^2}(W)$ is a $p$-group, and normalizes $V_2$, is nilpotent.

Iterating this process a finite number of times, we end up with a
$p$-local finite group $(S,\tilde{\Ff},\tilde{\Ll})$ for which $V$
is $\tilde{\Ff}$-normal and such that $(S,\Ff,\Ll)$ is nilpotent if and only if
$(S,\tilde{\Ff},\tilde{\Ll})$ is so. We conclude by
Lemma~\ref{blackburn}.
\end{proof}

Note that $p>2$ is a necessary condition, since the $2$-local
finite group induced by the semidirect product $Q_8\rtimes 3$
provides a counterexample at the prime~$2$.

\section{Cohomological characterizations in high degrees}
\label{sec highdegrees}

In contrast to the results in Section~\ref{sectionlow}, we look now
at cohomological characterizations of nilpotent $p$-local finite
groups in high degrees. The proofs follow the lines of Quillen's
arguments in the case of finite groups \cite{MR0318339}. Quillen
attributes the first criterion to Atiyah. Let $\Ff$ be a saturated
fusion system. We will use the ring of stable elements:
\[ H^*(\Ff; \F_p) = \underset{\overleftarrow{\Or (\Ff)}}{\lim } \;
H^*(-;\F_p), \]
as introduced in \cite[Section~5]{MR1992826}. It only depends on the
fusion system, but the main theorem of the cited section identifies
it with the mod $p$ cohomology of $|\Ll|$ when a centric linking
system $\Ll$ associated to $\Ff$ exists.

\begin{theorem}
\label{quillen2}
Let $(S,\Ff,\Ll)$ be a $p$-local finite group and let $Bi: BS
\rightarrow |\Ll|$ be the standard inclusion. Then $(S,\Ff,\Ll)$
is nilpotent if and only if one of the following three conditions is satisfied:

\begin{itemize}
\item[(1)] $Bi^*: H^n(|\Ll|; \F_p) \rightarrow H^n(BS; \F_p)$ is
an isomorphism for all sufficiently large~$n$.

\item[(2)] $Bi^*: H^n(|\Ll|; \Z^\wedge_p) \rightarrow
H^n(BS; \Z^\wedge_p)$ is an isomorphism for all sufficiently
large~$n$.

\item[(3)] For each $x\in H^{\text{even}}(BS; \F_p)$ there exists
a power $q$ of $p$ such that $x^q\in\im (Bi^*)$ ($p>2$).
\end{itemize}

\end{theorem}

\begin{proof}
Conditions (1) and (2) are equivalent by a universal coefficient
theorem argument. We work therefore with $p$-adic coefficients and
assume now that the $p$-local finite group $(S,\Ff,\Ll)$ satisfies
condition~(2). By Castellana and Morales' theorem \cite{CM},
$K^0(|\Ll|; \Z^\wedge_p)$ is a free $\Z^\wedge_p$-module and its
rank is the number of $\Ff$-conjugacy classes of elements in~$S$.
However, the map $Bi: BS \rightarrow |\Ll|$ induces a morphism of
Atiyah-Hirzebruch spectral sequences converging to the $p$-completed
$K$-theory rings $K^*(|\Ll|; \Z^\wedge_p) \rightarrow K^*(BS;
\Z^\wedge_p)$. This map has finite kernel and cokernel by
assumption, so that $K^0(|\Ll|; \Z^\wedge_p)$ and $K^0(BS;
\Z^\wedge_p)$ must have the same rank. Two elements in $S$ are hence
$\Ff$-conjugate if and only if they are $S$-conjugate and this is
the characterization of nilpotence given by part~(1) of Theorem~\ref{global}.

To prove that condition~(3) characterizes nilpotency for $p>2$, we
show that it implies the condition appearing in
Theorem~\ref{quillen1}. Let $V \triangleleft S$ be an elementary
abelian normal subgroup of $S$ and consider the ideal
\[ \p _V = \{ u \in H^{\text{even}}(BS;\F_p) \mid u|_{H^{\text{even}}(BV;\F_p)} \text{ is nilpotent} \}
\]
in $H^{\text{even}}(BS;\F_p)$. It is a prime ideal since restriction to $V$ induces an injection
$ H^{\text{even}}(BS;\F_p)/\p _V \to H^{\text{even}}(BV;\F_p)/\sqrt{0} = S(V^{\#}) $, where
$S(V^{\#})$ is the symmetric algebra on $V^{\#}= \Rep(V,\Z /p) $ with
the elements of $V^{\#}$ in degree $2$ \cite{MR0318339}. For simplicity of notation, let
us denote by $Bi^{-1}(X)$ the set $(Bi^*)^{-1}(X)$ for any $X \subseteq H^{\text{even}}(BS;\F_p) $.

Let $A$ be an elementary abelian subgroup of $S$ which is
$\Ff$-isomorphic to $V$. Then we have $ Bi^{-1}(\p _A) = Bi^{-1}(\p
_V) $, since $H^*(|\Ll |; \F_p)$ is computed by stable elements. Given $ u
\in \p _V $, condition~(3) implies that $ u^q \in \im (Bi^*)$ for some $q$, power of $p$. Say
$Bi^*(v) = u^q $. That means $ v \in Bi^{-1}(\p _V) = Bi^{-1}(\p _A)
$ and so $ u^q \in \p _A$. But $ \p _A $ is a prime ideal of
$H^{\text{even}}(BS; \F_p)$, so $ u \in \p _A$. We conclude that $ \p _A = \p _V $.
Applying \cite[Theorem 2.7]{MR0318339} we obtain that $A$ and $V$ must
be conjugate subgroups in $S$. Since $ V \triangleleft S $, $ A = V$
and $ V$ is $\Ff$-weakly closed.

It remains to prove that $\aut_{\Ff}(V)$ is a $p$-group. We will actually
show that $\aut_{\Ff}(V) = \aut_S(V) $. We proceed as Quillen in
\cite[Theorem~2.10]{MR0318339}. Consider the maps $ H^{\text{even}}(|\Ll |;\F_p) /
Bi^{-1}(\p _V) \longrightarrow H^{\text{even}}(BS;\F_p) / \p _V \longrightarrow
S(V^{\#}) $ and the associated extensions of quotient fields
\[  k(Bi^{-1}(\p _V)) \rightarrow k(\p _V) \rightarrow k(V). \]

Note that the groups of automorphisms of the extensions $k(\p _V) \rightarrow k(V)$
and $ k(Bi^{-1}(\p _V)) \rightarrow k(V) $ are $\aut _S(V) $ and
$\aut _{\Ff}(V) $, respectively. The proof of \cite[Proposition 5.2]{MR1992826}
shows that the ring $H^{\text{even}}(BS;\F_p)$
is integral over $H^{\text{even}}(|\Ll|;\F_p)$ and thus $ k(Bi^{-1}(\p _V)) \rightarrow k(\p _V) $ is
an algebraic extension. Since any element in $H^{\text{even}}(BS;\F_p)$
is in $H^{\text{even}}(\Ll;\F_p)$ after being raised to a certain power of $p$,
the extension $ k(Bi^{-1}(\p _V)) \rightarrow k(\p _V) $ is purely
inseparable, \cite[Theorem 19.10]{MR1276273}. Therefore, $ \aut_{\Ff}(V) = \aut_S(V)$.
\end{proof}

We also offer a characterization of nilpotency in terms of the
Morava $K$-theory $K(n)$, for any $n \geq 1$. For honest groups
this criterion has been discovered by Brunetti in
\cite{MR1618923}. His proof is based on the beautiful result
\cite{MR1758754} of Hopkins, Kuhn, and Ravenel on generalized
characters of finite groups.

\begin{theorem}
\label{brunetti}
Let $(S,\Ff,\Ll)$ be a $p$-local finite group and let $Bi: BS
\rightarrow |\Ll|$ be the standard inclusion. Then $(S,\Ff,\Ll)$ is
nilpotent if and only if $Bi^*: K(n)^*(|\Ll|) \rightarrow
K(n)^*(BS)$ is an isomorphism.
\end{theorem}

\begin{proof}
By \cite{MR1718090}, if we have a map of spaces inducing an
isomorphism on the $n$th Morava K-theory, then it also induces an
isomorphism on the first Morava K-theory. By \cite{CM}, an
isomorphism in $K(1)$ implies that two elements are conjugate in $S$
if and only if they are conjugate in $\Ff$, that is, condition~(1)
in Theorem~\ref{lowdegree}.

Conversely, if a $p$-local finite group is nilpotent, we have seen
in Proposition~\ref{prop cohomology} that $Bi$ induces an
isomorphism in mod $p$ cohomology. It must be then a
$K(n)$-equivalence as well, see for example
\cite[Corollary~1.5]{MR0467732}.
\end{proof}

\section{Quillen categories}
\label{sec Quillencat}

For a finite $p$-group $S$, let $ \varepsilon _S $ denote the
category whose objects are the elementary abelian $p$-subgroups of
$S$ and whose morphisms are given by conjugation. Similarly, for a
$p$-local finite group $(S,\Ff,\Ll)$, let $ \varepsilon _{\Ff} $ be
the category with the same objects considered as a full subcategory
of $\Ff$. They are the Quillen categories of $S$ and $\Ff$,
respectively. Recall that a ring homomorphism $ \gamma : B \to A $
is called an F-isomorphism if each element in $\text{Ker}(\gamma)$
is nilpotent and if for all $a \in A$ there is some $ k > 0 $ such
that $ a^k \in \text{Im}(\gamma)$. Following \cite{MR1992826}, we
define

\[ H^*_{\varepsilon} (S;\F_p) = \underset{\overleftarrow{\varepsilon _S}}{\lim } \; H^*(-;\F_p) \ \ \hbox{\rm and} \ \ H^*_{\varepsilon} (\Ff;\F_p) = \underset{\overleftarrow{\varepsilon _{\Ff}}}{\lim } \; H^*(-;\F_p). \]

\begin{definition}
\label{isotypical}
{\rm Let $\mathcal C$ and $\mathcal D$ be two categories equipped
with functors $\mathcal C \stackrel{\gamma}{\rightarrow} Gr $ and $
\mathcal D \stackrel{\delta}{\rightarrow} Gr $ to the category of
groups. A functor $ \Psi : \mathcal C \to \mathcal D $ is
\emph{isotypical} if $ \gamma $ is naturally isomorphic to $ \delta
\circ \Psi $. }
\end{definition}

\begin{theorem}
\label{quillen3?}
Let $(S,\Ff,\Ll)$ be a $p$-local finite group, $p>2$, and let $Bi: BS
\rightarrow |\Ll|$ be the standard inclusion. Then $(S,\Ff,\Ll)$ is
nilpotent if and only if one of the following conditions is
satisfied:

\begin{itemize}
\item[(1)] The Quillen categories of $S$ and $\Ff$ are
isotypically equivalent.

\item[(2)] $Bi^*: H^*(|\Ll|; \F_p) \rightarrow H^*(BS; \F_p)$
is an $F$-isomorphism.
\end{itemize}
\end{theorem}

\begin{proof}
By \cite[Proposition 5.1]{MR1992826}, there are F-isomorphisms $
\lambda _S : H^*(BS) \to H_{\varepsilon} ^*(S) $ and $ \lambda
_{\Ff} : H^*(\Ff) \to H_{\varepsilon} ^*(\Ff) $.  The stable element
formula \cite[Theorem~5.8]{MR1992826} shows that the natural map $
R_{\Ll} : H^*(|\Ll|) \to H^*(\Ff) $ is an isomorphism. So in fact,
we have a commutative diagram:
\[
\diagram H^*(|\Ll| ) \rto^{\lambda _{\Ff} R_{\Ll} } \dto_{Bi^*} &
H_{\varepsilon} ^*(\Ff) \dto^{j} \cr H^*(BS) \rto^{\lambda _S} &
H_{\varepsilon} ^*(S).
\enddiagram
\]

If the Quillen categories of $S$ and $\Ff$ are isotypically
equivalent, then the map $j$ is an isomorphism and therefore the map
$ Bi^* $ is an $F$-isomorphism. Assume now that $Bi^*$ is an
$F$-isomorphism. Then condition~(3) in Theorem~\ref{quillen2} holds
and so $(S,\Ff,\Ll)$ is nilpotent.
\end{proof}

Let $\Ca$ be a class of finite subgroups of $S$. We say that
a subgroup $H$ of $S$ \emph{controls fusion} of $\Ca$-groups in
$\Ff$ if the following conditions hold:

\begin{itemize}
\item Any $\Ca$-subgroup of $S$ is $\Ff$-conjugate to a subgroup of $H$.
\item For any $\Ca$-subgroup $P$ of $S$ and any $ f : P \to S $ in $\Ff$ such that
$ f(P) \leq H $, there exists $ h \in H $ such that $ f(x) =
hxh^{-1} $ for all $x \in P$.
\end{itemize}

When $H=S$, the first condition is a tautology, so $S$ itself
controls fusion if any $\Ff$-morphism $P\rightarrow S$ from a group
$P \in \Ca$ is conjugation by an element in~$S$.

\begin{definition}
\label{quasicyclic}
{\rm Let $G$ be a group. We say $ x \in G $ is a $\Ca _p$-\emph{element} when $x^p = 1$ if
$p$ is odd or $ x^4 = 1 $ if $p = 2$. A subgroup of $G$ generated by a $\Ca _p$-element
is called a $\Ca_p$-\emph{subgroup} and we denote by $\Ca_p$ the class of $\Ca_p$-subgroups. }
\end{definition}

Given $K \leq S $ and $ f:K \to S$, let us denote by $[K,f] $ the subgroup of $S$ generated by
the elements of the form $[a,f]=af(a)$ with $a \in K $. We also denote by $[K,f,f]$ the subgroup
$[[K,f],f] $.

\begin{theorem}
\label{control}
Let $(S,\Ff,\Ll)$ be a p-local finite group. Then $(S,\Ff,\Ll)$ is
nilpotent if and only if $S$ controls fusion of $\Ca_p$-subgroups.
\end{theorem}

\begin{proof}
We follow the strategy in \cite[Theorem~2]{G}. We will show that the
control of fusion condition implies condition~(3) in
Theorem~\ref{global}. Let $P \leq S $, and $ f \in \aut _{\Ff}(P) $
be an automorphism of order prime to $p$. There exists some $l \geq 1$
such that $P$ is contained in a subgroup $Z_{l}(S)$ of the upper
central series. We will prove by induction on $l$ that $ f $ is the
identity map.

Suppose first that $ P \leq Z(S) $ and consider a $\Ca_p$-element
$a \in P $. Since $S$ controls fusion of $\Ca_p$-subgroups, there is
an $s \in S $ such that $ f(a) = sas^{-1}$. But $ a \in Z(S) $, so
$ f(a) = a $. By \cite[Theorem~5.2.4]{MR569209}, $ f = 1_P $.

Now consider $ P \leq Z_l(S)$ and suppose that the result is known
for subgroups of $ Z_{l-1}(S)$. Let $ K $ be the subgroup of $P$
generated by $\Ca_p$-elements. Both $K$ and $[K,f]$ are
stabilized by $f$.

Consider a $\Ca_p$-element $ a \in P $. Since $S$ controls
fusion of $\Ca_p$-subgroups, there is $s \in S $ such that $
f(a) = sas^{-1} $ and so $[a,f] = [a,s] \in Z_{l-1}(S)$. Note that
$[K,f]$ is generated by elements of the form $ b[a,f]b^{-1} $, where
$a$, $b \in K $ and $ a $ is a $\Ca_p$-element. Therefore $[K,f] \leq
Z_{l-1}(S)$, so that the induction hypothesis applies: $ f $
restricted to $ [K,f]$ must be the identity. Since $f$ stabilizes
$[K,f]$, $ [K,f,f] = 1 $. But by \cite[Theorem 5.3.6]{MR569209}, $
[K,f,f]=[K,f]$ under these circumstances, so $ f $ restricted to $K$
is the identity as well. In particular, $f$ fixes the $\Ca_p$-elements
of $P$. Now \cite[Satz~IV.5.12]{MR0224703} implies that $f$ is the identity.
\end{proof}

\begin{definition}{\rm (\cite{MR2167090})}
{\rm For any saturated fusion system $\Ff$ over a finite $p$-group
$S$, the \emph{center} of $\Ff$ is the subgroup
\[ Z_{\Ff}(S) = \{ x \in Z(S) \mid f(x) = x \ \forall f \in \Mor(\Ff
^c) \} = \underset{\overleftarrow{\Ff ^c}}{\lim } \; Z(-) \] }
\end{definition}

The characteristic subgroup $\Omega_i(S)$ is the subgroup generated
by all elements $x$ such that $x^{p^i} =1$. The following result is
now a straightforward consequence of Theorem~\ref{control}.

\begin{corollary}
Let $(S,\Ff,\Ll)$ be a $p$-local finite group. If the $\Ca_p$-elements
of $S$ are in the center of $\Ff$, then $(S,\Ff,\Ll)$ is nilpotent.
In particular, if $p$ is odd and $\Omega _1 (S) \subseteq
Z _{\Ff}(S)$ (respectively $ p = 2 $ and $\Omega _2 (S) \subseteq Z
_{\Ff}(S)$), then $(S,\Ff,\Ll)$ is nilpotent.
\end{corollary}


\bibliographystyle{alpha}\label{biblio}
\bibliography{nilpotent}

\newcommand{\etalchar}[1]{$^{#1}$}
\begin{thebibliography}{DGMP09}

\bibitem[Alp64]{MR0167528}
J.~L. Alperin.
\newblock Centralizers of abelian normal subgroups of {$p$}-groups.
\newblock {\em J. Algebra}, 1:110--113, 1964.

\bibitem[BCG{\etalchar{+}}05]{MR2167090}
C.~Broto, N.~Castellana, J.~Grodal, R.~Levi, and B.~Oliver.
\newblock Subgroup families controlling {$p$}-local finite groups.
\newblock {\em Proc. London Math. Soc. (3)}, 91(2):325--354, 2005.

\bibitem[BCG{\etalchar{+}}07]{MR2302515}
C.~Broto, N.~Castellana, J.~Grodal, R.~Levi, and B.~Oliver.
\newblock Extensions of {$p$}-local finite groups.
\newblock {\em Trans. Amer. Math. Soc.}, 359(8):3791--3858 (electronic), 2007.

\bibitem[BK72]{MR51:1825}
A.~K. Bousfield and D.~M. Kan.
\newblock {\em Homotopy limits, completions and localizations}.
\newblock Springer-Verlag, Berlin, 1972.
\newblock Lecture Notes in Mathematics, Vol. 304.

\bibitem[Bla66]{MR0190238}
N.~Blackburn.
\newblock Automorphisms of finite {$p$}-groups.
\newblock {\em J. Algebra}, 3:28--29, 1966.

\bibitem[BLO03a]{MR1961340}
C.~Broto, R.~Levi, and B.~Oliver.
\newblock Homotopy equivalences of {$p$}-completed classifying spaces of finite
  groups.
\newblock {\em Invent. Math.}, 151(3):611--664, 2003.

\bibitem[BLO03b]{MR1992826}
C.~Broto, R.~Levi, and B.~Oliver.
\newblock The homotopy theory of fusion systems.
\newblock {\em J. Amer. Math. Soc.}, 16(4):779--856 (electronic), 2003.

\bibitem[Bru98]{MR1618923}
M.~Brunetti.
\newblock A new cohomological criterion for the {$p$}-nilpotence of groups.
\newblock {\em Canad. Math. Bull.}, 41(1):20--22, 1998.

\bibitem[BS08]{MR2378355}
D.J. Benson and S.D. Smith.
\newblock {\em Classifying spaces of sporadic groups}, volume 147 of {\em
  Mathematical Surveys and Monographs}.
\newblock American Mathematical Society, Providence, RI, 2008.

\bibitem[CM10]{CM}
N.~Castellana and D.~Morales.
\newblock Generalized cohomology theories for $p$-local finite groups.
\newblock Preprint, 2010.

\bibitem[DGMP09]{MR2448569}
A.~D{\'{\i}}az, A.~Glesser, N.~Mazza, and S.~Park.
\newblock Glauberman's and {T}hompson's theorems for fusion systems.
\newblock {\em Proc. Amer. Math. Soc.}, 137(2):495--503, 2009.

\bibitem[DGPS10]{DGPS}
A.~D{\'{\i}}az, A.~Glesser, S.~Park, and R.~Stancu.
\newblock Tate's and Yoshida's theorem on control of transfer for fusion
  systems.
\newblock {\em To appear in Journal of the London Math.\ Soc.}

\bibitem[Far96]{MR1392221}
E.~D. Farjoun.
\newblock {\em Cellular spaces, null spaces and homotopy localization}, volume
  1622 of {\em Lecture Notes in Mathematics}.
\newblock Springer-Verlag, Berlin, 1996.

\bibitem[Gle10]{GL}
A.~Glesser.
\newblock Sparse fusion system.
\newblock {\em to appear in Proceedings of the Edinburgh Math.\ Soc.}

\bibitem[Gor80]{MR569209}
D.~Gorenstein.
\newblock {\em Finite groups}.
\newblock Chelsea Publishing Co., New York, second edition, 1980.

\bibitem[GS10]{G}
J.~Gonzalez-Sanchez.
\newblock A $p$-nilpotency criterion.
\newblock {\em Arch. Math.}, 94(3):201--205, 2010.

\bibitem[HKR00]{MR1758754}
M.~J. Hopkins, Nicholas~J. Kuhn, and D.~C. Ravenel.
\newblock Generalized group characters and complex oriented cohomology
  theories.
\newblock {\em J. Amer. Math. Soc.}, 13(3):553--594 (electronic), 2000.

\bibitem[HP94]{MR1289332}
H.W. Henn and S.~Priddy.
\newblock {$p$}-nilpotence, classifying space indecomposability, and other
  properties of almost all finite groups.
\newblock {\em Comment. Math. Helv.}, 69(3):335--350, 1994.

\bibitem[Hup67]{MR0224703}
B.~Huppert.
\newblock {\em Endliche {G}ruppen. {I}}.
\newblock Die Grundlehren der Mathematischen Wissenschaften, Band 134.
  Springer-Verlag, Berlin, 1967.

\bibitem[Isa94]{MR1276273}
I.~Martin Isaacs.
\newblock {\em Algebra}.
\newblock Brooks/Cole Publishing Co., Pacific Grove, CA, 1994.
\newblock A graduate course.

\bibitem[KL03]{MR1929020}
R.~Kessar and M.~Linckelmann.
\newblock A block theoretic analogue of a theorem of {G}lauberman and
  {T}hompson.
\newblock {\em Proc. Amer. Math. Soc.}, 131(1):35--40 (electronic), 2003.

\bibitem[KL08]{MR2379788}
R.~Kessar and M.~Linckelmann.
\newblock {$ZJ$}-theorems for fusion systems.
\newblock {\em Trans. Amer. Math. Soc.}, 360(6):3093--3106, 2008.

\bibitem[Lin07]{MR2336638}
M.~Linckelmann.
\newblock Introduction to fusion systems.
\newblock In {\em Group representation theory}, pages 79--113. EPFL Press,
  Lausanne, 2007.

\bibitem[LO02]{MR1943386}
R.~Levi and B.~Oliver.
\newblock Construction of 2-local finite groups of a type studied by {S}olomon
  and {B}enson.
\newblock {\em Geom. Topol.}, 6:917--990 (electronic), 2002.

\bibitem[Mis78]{MR0467732}
G.~Mislin.
\newblock Localization with respect to {$K$}-theory.
\newblock {\em J. Pure Appl. Algebra}, 10(2):201--213, 1977/78.

\bibitem[Oli04]{MR2092063}
B.~Oliver.
\newblock Equivalences of classifying spaces completed at odd primes.
\newblock {\em Math. Proc. Cambridge Philos. Soc.}, 137(2):321--347, 2004.

\bibitem[Oli06]{MR2203209}
B.~Oliver.
\newblock Equivalences of classifying spaces completed at the prime two.
\newblock {\em Mem. Amer. Math. Soc.}, 180(848):vi+102, 2006.

\bibitem[Pui01]{Puig}
Ll. Puig.
\newblock Full frobenius systems and their localizing categories.
\newblock Preprint, 2001.

\bibitem[Qui71]{MR0318339}
D.~Quillen.
\newblock A cohomological criterion for {$p$}-nilpotence.
\newblock {\em J. Pure Appl. Algebra}, 1(4):361--372, 1971.

\bibitem[Rob82]{MR648604}
D.J.S.\ Robinson.
\newblock {\em A course in the theory of groups}, volume~80 of {\em Graduate
  Texts in Mathematics}.
\newblock Springer-Verlag, New York, 1982.

\bibitem[Sta77]{MR0473039}
U.~Stammbach.
\newblock Another homological characterisation of finite {$p$}-nilpotent
  groups.
\newblock {\em Math. Z.}, 156(2):209--210, 1977.

\bibitem[Tat64]{MR0160822}
J.~Tate.
\newblock Nilpotent quotient groups.
\newblock {\em Topology}, 3(suppl. 1):109--111, 1964.

\bibitem[Wil99]{MR1718090}
W.~S. Wilson.
\newblock {$K(n+1)$} equivalence implies {$K(n)$} equivalence.
\newblock In {\em Homotopy invariant algebraic structures ({B}altimore, {MD},
  1998)}, volume 239 of {\em Contemp. Math.}, pages 375--376. Amer. Math. Soc.,
  Providence, RI, 1999.

\end{thebibliography}


\end{document}